\documentclass[11pt, reqno]{amsart}
\usepackage{amsfonts}
\usepackage{amsmath}
\usepackage{amssymb}
\usepackage{amscd}
\usepackage[mathscr]{eucal}
\usepackage{url}

\allowdisplaybreaks[1]

\newcommand{\aN}[4]{a\left(\begin{smallmatrix} #1 & #2 \\ #3 & #4 \end{smallmatrix}\right)}
\newcommand{\sm}[4]{\left(\begin{smallmatrix} #1 & #2 \\ #3 & #4 \end{smallmatrix}\right)}
\newcommand{\mat}[4]{\begin{pmatrix} #1 & #2 \\ #3 & #4 \end{pmatrix}}

\makeatletter
\def\imod#1{\allowbreak\mkern2mu({\operator@font mod}\,\,#1)}
\makeatother

\makeatletter
\def\jmod#1{\allowbreak\mkern5mu({\operator@font mod}\,\,#1)}
\makeatother

\theoremstyle{plain}

\newtheorem{theorem}{Theorem}[section]
\newtheorem{lemma}[theorem]{Lemma}
\newtheorem{prop}[theorem]{Proposition}
\newtheorem{cor}[theorem]{Corollary}
\theoremstyle{definition}

\numberwithin{equation}{section}

\makeatletter
\def\imod#1{\allowbreak\mkern5mu({\operator@font mod}\,\,#1)}
\makeatother

\begin{document}

\title[Multiplicative Relations for Coefficients of Degree 2 Eigenforms]{Multiplicative Relations for Fourier Coefficients\\ of Degree 2 Siegel Eigenforms}

\author{Dermot M\lowercase{c}Carthy}
\address{Dermot M\lowercase{c}Carthy, Department of Mathematics \& Statistics\\
Texas Tech University\\
Lubbock, TX 79410-1042\\
USA}
\email{dermot.mccarthy@ttu.edu}




\begin{abstract}
We prove multiplicative relations between certain Fourier coefficients of degree 2 Siegel eigenforms. These relations are analogous to those for elliptic eigenforms. We also provide two sets of formulas for the eigenvalues of degree 2 Siegel eigenforms. The first evaluates the eigenvalues in terms of the form's Fourier coefficients, in the case $a(I) \neq 0$. The second expresses the eigenvalues of index $p$ and $p^2$, for $p$ prime, solely in terms of $p$ and $k$, the weight of the form, in the case $a(0)\neq 0$. From this latter case, we give simple expressions for the eigenvalues associated to degree 2 Siegel Eisenstein series. 
\end{abstract}

\maketitle


\section{Introduction and Statement of Main Results}\label{sec_Intro}
The theory of Hecke operators provides us with many of the fundamental results about the spaces of elliptic modular forms. For example, the space of elliptic modular forms, of a given weight, has a basis of Hecke eigenforms which have multiplicative Fourier coefficients. 


Hecke theory has been extended to Siegel modular forms, with the work of Andrianov at its core (see for example \cite{AN, AN2, AN3}), but in some respects the results aren't as satisfying. While the spaces of Siegel modular forms have a basis of eigenforms, and Andrianov provides us with a comprehensive structure for the relationship between the eigenvalues and Fourier coefficients of degree 2 eigenforms, we do not get simple multiplicative relations between the Fourier coefficients, as exists in the elliptic case. The main purpose of this paper is to prove, using the results of Andrianov, that simple multiplicative relations, which are analogous to the elliptic case, do exist between certain Fourier coefficients of degree 2 Siegel eigenforms.  

Let $f$ be an elliptic eigenform of weight $k\geq1$ on the full modular group, with Fourier expansion $f(z)= \sum_{n \geq 0} a(n) q^n$, with $q:=e^{2 \pi i z}$. Then its Fourier coefficients satisfy the following properties \cite{BGHZ, DS}:
\begin{enumerate}
\item If $(1)=0$, then $a(m)=0$ for all $m \in \mathbb{Z^{+}}$.\\[-9pt]
\item $a(1) \, a(mn) = a(m) \, a(n)$ when $\gcd(m,n)=1$.\\[-9pt]
\item $a(1) \, a(p^{r+1}) = a(p) \, a(p^r) - p^{k-1} \, a(1) \, a(p^{r-1})$, for all $p$ prime and $r \geq 1$.
\end{enumerate} 
If $f$ is non-zero, then property (1) ensures that $a(1)\neq 0$, and we can normalize $f$ by setting $a(1)=1$. In which case, we can drop the $a(1)$ factors in properties (2) and (3). 

The main result of this paper provides analogous properties for degree 2 Siegel eigenforms. Let $M_k^2(\Gamma)$ denote the space of Siegel modular forms of degree 2 and weight $k$ on the full modular group. Let $I$ denote the $2 \times 2$ identity matrix.\\


\begin{theorem}\label{thm_mult}
Let $F(Z) = \sum_{N \geq 0} a(N) \, exp(2 \pi i \, \textup{Tr}(NZ)) \in M_k^2(\Gamma)$ be an eigenform. 
\begin{enumerate}
\item If $a \left(I \right) =0$, then $a \left(mI \right) =0$ for all $m \in \mathbb{Z^{+}}$.\\[-9pt]
\item $a \left(I \right) \, a \left(mnI \right) = a \left(mI \right) \, a \left(nI \right)$ when $\gcd(m,n)=1$.\\[-9pt]
\item $a\left(I \right)  \, a \left(p^{r+1} I \right)  = a \left(pI \right)  \, a \left(p^{r} I \right) - p^{2k-3} \, a \left(I \right)  \, a \left(p^{r-1} I \right) \\
- p^{k-2} \, a \left(I \right)
 \Biggl[ 2 \, \aN{p^{r-1}}{0}{0}{p^{r+1}} + (1+ (-1)^k) \displaystyle \sum^{\frac{p-1}{2}}_{\substack{u = 1 \\ u^2 \not\equiv -1 \imod{p}}} a\left( p^r \left(\begin{smallmatrix} (1+u^2)p^{-1} & u \\ u & p \end{smallmatrix} \right) \right)  \Biggr]$,\\ for all $p$ prime and $r \geq 1$, and where the last sum is vacuous in the case $p=2$.
\end{enumerate} 
\end{theorem} 

If $a(I) \neq 0$, we can normalize $F$ by setting $a(I) = 1$. In this case we can remove the $a(I)$ factors from properties (2) and (3) of Theorem \ref{thm_mult}. Unfortunately, property (1) of Theorem \ref{thm_mult} is not sufficient to ensure $a(I) \neq 0$ for non-zero eigenforms, as happens in the case of elliptic eigenforms. In fact, if the weight $k$ is odd then $a(I) = 0$, which we will see in Section \ref{sec_Prelim}. However, the following classes of degree 2 Siegel eigenforms all have $a(I) \neq 0$: Siegel Eisenstein series, $E_k$, which have even weight $k \geq 4$  (see Section 2); the unique cusp eigenforms of weights 10 and 12, often denoted $\chi_{10}$ and $\chi_{12}$, which along with the Eisenstein series $E_4$ and $E_6$ generate the graded ring of even weight degree 2 Siegel modular forms \cite{Ig1}; and Klingen Eisenstein series generated from elliptic eigenforms \cite{Mi}.

In the elliptic case, the eigenvalues of an eigenform, normalized with a(1)=1, are related to the form's Fourier coefficients in the following way:
$$\lambda(m) = a(m) \qquad \textup{for all } m \in \mathbb{Z}^{+},$$
where $\lambda(m)$ is the eigenvalue associated to the Hecke operator of index $m$. We can form a Dirichlet series from the eigenvalues $\lambda(m)$, and it has the following Euler product: 
$$ \sum_{m=1}^{\infty} \frac{\lambda(m)}{m^s} = \prod_p (1- \lambda(p) p^{-s} + p^{k-1-2s})^{-1}.$$
Therefore we can generate all eigenvalues if we know $\lambda(p)=a(p)$ for all primes $p$.
Analogously, in the degree 2 case, the Andrianov $L$-function of an eigenform, from which we can calculate the Dirichlet series formed from its eigenvalues, can be generated by $\lambda(p)$ and $\lambda(p^2)$, its eigenvalues of index $p$ and $p^2$ respectively. We prove the following theorems which describe how these eigenvalues can be calculated from the form's Fourier coefficients.

\begin{theorem}\label{thm_lambdaAI}
Let $F(Z) = \sum_{N \geq0} a(N) \exp \left({2 \pi i} \, \textup{Tr}(NZ) \right) \in M_k^2(\Gamma)$ be an eigenform, normalized with $a(I)=1$. Let
$$h_1(p) =
\begin{cases}
2 & \textup{ if } p \equiv 1 \pmod{4}\\
1  & \textup{ if } p=2\\
0  & \textup{ if } p \equiv 3 \pmod{4},
\end{cases}
\quad \textup{and} \quad
h_2(p) = 
\begin{cases}
2 & \textup{ if } p \equiv 1 \pmod{4}\\
0  & \textup{ if } p=2 \textup{ or }p \equiv 3 \pmod{4}.
\end{cases}$$
Then for any prime p, the eigenvalues of index $p$ and $p^2$ associated to $F$ satisfy
\begin{equation*}
\lambda_F(p) = a(pI) + h_1(p) \,p^{k-2}
\end{equation*}
and
\begin{equation*}
\lambda_F(p^2) = a(p^2I) + h_1(p) \, p^{k-2} \, a(pI) + h_2(p) \, p^{2k-4}.
\end{equation*}
\end{theorem}
\noindent Formulas for $\lambda_F(p)$ and $\lambda_F(p^2)$ in terms of $a\sm{1}{1/2}{1/2}{1}$ can be found in \cite{Sk}. 

As noted above, not all degree 2 Siegel eigenforms can be normalized with $a(I)=1$, so we have also evaluated the relevant eigenvalues in the case $a(0)\neq 0$.
\begin{theorem}\label{thm_lambdaA0}
Let $F(Z) = \sum_{N \geq 0} a(N) \exp \left({2 \pi i} \, \textup{Tr}(NZ) \right) \in M_k^2(\Gamma)$ be an eigenform with $a(0)\neq0$. Then for any prime p, the eigenvalues of index $p$ and $p^2$ associated to $F$ satisfy
\begin{equation*}
\lambda_F(p) = 1 +p^{k-1}+p^{k-2}+p^{2k-3},
\end{equation*}
and
\begin{equation*}
\lambda_F(p^2) = 1 +p^{k-2} (p+1) + p^{2k-4}(p^2+2p)+ p^{3k-5}(p+1) + p^{4k-6}.
\end{equation*}
\end{theorem}
\noindent It is well known (see section 2) that Siegel Eisenstein series are eigenforms and have $a(0)=1$ so the following corollary is immediate. 
\begin{cor}
Let $E_k$ be the Siegel Eisenstein series of weight $k$.
Then for any prime p, the eigenvalues of index $p$ and $p^2$ associated to $E_k$ satisfy
\begin{equation*}
\lambda_{E_k}(p) = 1 +p^{k-1}+p^{k-2}+p^{2k-3},
\end{equation*}
and
\begin{equation*}
\lambda_{E_k}(p^2) = 1 +p^{k-2} (p+1) + p^{2k-4}(p^2+2p)+ p^{3k-5}(p+1) + p^{4k-6}.
\end{equation*}
\end{cor}
\noindent The eigenvalues associated to Siegel Eisenstein series have previously been calculated by Walling in \cite{Wa} where the above formula for $\lambda_{E_k}(p)$ is a special case of Proposition 3.3 of \cite{Wa}. We note that Walling chooses to work with the operator $T_1(p^2)$ instead of $T(p^2)$ thus evaluating $\lambda_1(p^2)$ instead of $\lambda(p^2)$ as we have done.


\section{Siegel Modular Forms, Hecke Operators and Statement of Other Results}\label{sec_Prelim}

We start this section with a brief introduction to Siegel modular forms. Please see \cite{BGHZ,Kl} for further details. Let $\mathbb{A}^{m\times n}$ denote the set of all $m \times n$ matrices with entries in the set $\mathbb{A}$.  For a matrix $M$ we let ${^tM}$ denote its transpose; if $M$ is square, $\textup{Tr}(M)$ its trace and $\textup{Det}(M)$ its determinant; and if $M$ has entries in $\mathbb{C}$, $\textup{Im}(M)$ its imaginary part.
If a matrix $M \in \mathbb{R}^{n \times n}$ is positive definite, then we write $M > 0$, and if $M$ is positive semi-definite, we write $M \geq 0$.
The Siegel half-plane $\mathbb{H}^2$ of degree $2$ is defined by
$$\mathbb{H}^2:=  \left\{ Z \in \mathbb{C}^{2 \times 2} \mid {^tZ} = Z, \textup{Im}(Z)>0\right\}.$$
Let
$$\Gamma^2:=\textup{Sp}_{4}(\mathbb{Z})= \left\{ M \in \mathbb{Z}^{4 \times 4} \mid {^tM}JM = J \right\} ,\qquad
J=\sm{0}{I}{-I}{0},$$
be the Siegel modular group of degree $2$. We will often drop the superscript 2 if the degree is clear form the context.
The modular group $\Gamma^2$ acts on $\mathbb{H}^2$ via the operation
$$M \cdot Z = \left(AZ+B\right) \left(CZ+D\right)^{-1}$$
where
$M=\sm{A}{B}{C}{D} \in \Gamma^2$, $Z \in \mathbb{H}^2$.
A holomorphic function $F:\mathbb{H}^2 \to \mathbb{C}$ is called a Siegel modular form of degree $2$ and weight $k\in \mathbb{Z}^{+}$ on $\Gamma^2$ if
\[
F|_kM(Z):=\textup{Det}(CZ+D)^{-k} \, F(M \cdot Z)   =  F(Z)
\]
for all $M=\sm{A}{B}{C}{D} \in \Gamma^2$.  We note that the desired boundedness of $F|_k M(Z)$, for any $M \in \Gamma^2$, when $\textup{Im}(Z) - cI_2 \geq 0$, with fixed $c >0$, is automatically satisfied by the Koecher principle.
The set of all such modular forms is a finite dimensional vector space over $\mathbb{C}$, which we denote $M_k^2(\Gamma)$.
Every $F \in M_k^2(\Gamma)$ has a Fourier expansion of the form
$$F(Z) = \sum_{N \in \mathcal{R}^2} a(N) \exp \left({2 \pi i} \, \textup{Tr}(NZ) \right)$$
where $Z \in \mathbb{H}^2$ and
$$\mathcal{R}^2 = \left\{N=(N_{ij}) \in \mathbb{Q}^{2 \times 2} \mid {^tN} = N \geq 0, N_{ii}, 2 N_{ij} \in \mathbb{Z} \right\}.$$
We note that
\begin{equation}\label{for_UNTU}
a(UN{^tU})=\textup{Det}(U)^k \, a(N)
\end{equation} 
for $U \in GL_2(\mathbb{Z})$. In particular, if $U=\sm{0}{1}{-1}{0}$ then we get that 
\begin{equation}\label{for_ASwitch1}
\aN{r}{b/2}{b/2}{s} = \aN{s}{-b/2}{-b/2}{r}.
\end{equation}
Similarly, if $U=\sm{0}{1}{1}{0}$ then we get that 
\begin{equation*}\label{for_ASwitch2}
\aN{r}{b/2}{b/2}{s} = (-1)^k \, \aN{s}{b/2}{b/2}{r}.
\end{equation*}
Therefore, when $k$ is odd, $a(I)=0$.

We call $F \in M_k^2(\Gamma)$ a cusp form if $a(N)=0$ for all $N \not> 0$ and denote the space of such forms $S_k^2(\Gamma)$.

We define the Siegel Eisenstein series of weight $k$, denoted $E_k$, by
$$E_k(Z) := \sum_{(C,D)} \textup{Det}(CZ+D)^{-k}$$
where the sum is over non-associated, with respect to left multiplication by $GL_2(\mathbb{Z})$, pairs of coprime symmetric matrices $C, D \in \mathbb{Z}^{2 \times 2}.$ We note that $E_k \in M_k^2(\Gamma)$ for even $k \geq 4$ with Fourier coefficients \cite{EZ, Ko}
$$a_{k}(N) = a_{k} \sm{r}{b/2}{b/2}{s} = \frac{2}{\zeta(1-k) \zeta(3-2k) } \sum_{d \mid (r,b,s)} d^{k-1} \, H\left(k-1, \frac{4rs-b^2}{d^2}\right),$$
where $\zeta(\cdot)$ is the Riemann zeta function and $H(\cdot)$ is Cohen's generalized class number function \cite{Co}. In particular, $a_k(0)=1$ and $a_k(I) \neq 0$ for all even $k \geq 4$.

In \cite{AN}, Andrianov gives a very nice description of the Hecke theory for Siegel modular forms on the full modular group. We gave a brief summary here for the degree 2 case.

Let
$$S^{(2)}:= \left\{ M \in \mathbb{Z}^{4 \times 4} \mid {^tM}JM = r(M) \, J, \; r(M)=1,2,\cdots \right\} $$
Then every double coset $\Gamma M \Gamma$, with $M \in S^{(2)}$, can be written as union of finitely many right cosets of $\Gamma$ in $S^{(2)}$, i.e.,
$$\Gamma M \Gamma =  \bigcup_{i=1}^{\mu} \Gamma \sigma_i,$$
for some $\sigma_i \in S^{(2)}$, $\mu \in \mathbb{Z}^{+}$.
For each such double coset we associate an operator $T_k(\Gamma M \Gamma)$ which acts on $M_k^2(\Gamma)$ as follows. For $F \in M_k^2(\Gamma)$,
$$T_k(\Gamma M \Gamma) F := r(M)^{2k-3} \, \sum_{i=1}^{\mu} F |_k \sigma_i.$$
$T_k(\Gamma M \Gamma)$ is independent of the choice of representatives $\{ \sigma_i \}$ and maps $M_k^2(\Gamma)$ into itself. We call $F \in M_k^2(\Gamma)$ an eigenform if it is an eigenfunction for all the operators $T_k(\Gamma M \Gamma)$, $M \in S^{(2)}$. For all $k \geq 1$, $M_k^2(\Gamma)$ has a basis consisting of eigenforms. For even $k \geq 4$ the one-dimensional subspace of $M_k^2(\Gamma)$ generated by the Eisenstein series $E_k$ is invariant under the action of all the $T_k(\Gamma M \Gamma)$, and so the Eisenstein series $E_k$ are eigenforms.  

We now define the Hecke operator of index $m$ by the following finite sum:
$$T_k(m) := \sum_{r(M)=m} T_k(\Gamma M \Gamma).$$
Then
\begin{equation}\label{for_TCommute}
T_k(m) \, T_k(n) = T_k(n) \, T_k(m) = T_k(mn), \quad \textup{when } (m,n)=1.
\end{equation}
For $F \in M_k^2(\Gamma)$ an eigenform, we define its eigenvalues, $\lambda_F(m)$, by
$$T_k(m) \, F = \lambda_F(m) \, F.$$
We will refer to $\lambda_F(m)$ as the eigenvalue of index $m$ associated to $F$ and we note that these eigenvalues are real.

In \cite{AN}, Andrianov considers the Fourier coefficients of $T_k(m) \, F$, for $F \in M_k^2(\Gamma)$. Given (\ref{for_TCommute}) it suffices to study $T_k(p^{\delta}) \, F$, for $\delta \geq 1$. Let 
$$T_k(p^{\delta}) \, F (Z) = \sum_{N \in \mathcal{R}^2} a(p^{\delta}; N) \exp \left({2 \pi i} \, \textup{Tr}(NZ) \right).$$
Andrianov provides us with a formula for $a(p^{\delta}; N)$ in terms of $a(\cdot)$, the Fourier coefficients of $F$, which we state in Theorem \ref{thm_Andrianov} below. We first note that if $F$ is an eigenform, then for any $N \in \mathcal{R}^2$ we have the relation
\begin{equation}\label{eq_alambda}
a(N) \, \lambda(p^{\delta}) = a(p^{\delta}; N). 
\end{equation}
Let 
$$R(p^{\beta})=\left\{ \mat{u_1}{u_2}{u_3}{u_4} \in SL_2(\mathbb{Z}) \mid (u_1, u_2) \imod{p^{\beta}} \right\}$$
be any set of $2 \times 2$ integral matrices whose first row ranges over a complete set of representatives of the equivalence classes of relatively prime integers under the equivalence relation 
\begin{equation}\label{eq_EquivRel}
(u_1, u_2) \sim (u_1^{\prime}, u_2^{\prime}) \imod{p^{\beta}}
\Leftrightarrow au_1 \equiv u_1^{\prime}, au_2 \equiv u_2^{\prime} \imod{p^{\beta}},
\end{equation}
for some $a \in \left( \mathbb{Z} / p^{\beta} \mathbb{Z} \right)^{\times}$, and whose second rows are chosen so that $u_1 u_4 -u_2 u_3 = 1$.
For $N=\sm{r}{b/2}{b/2}{s}$, let $\sm{r_u}{b_u/2}{b_u/2}{s_u} = U N {^t}U$, for a given $U \in SL_2(\mathbb{Z})$.
\begin{theorem}[{Andrianov~\cite[(2.1.11)]{AN}}]\label{thm_Andrianov}
For $p$ prime and $N=\sm{r}{b/2}{b/2}{s}$, 
\begin{equation*}
a(p^{\delta}; N) =\sum_{\substack{\alpha+\beta+\gamma=\delta \\ \alpha, \beta, \gamma \geq 0}} p^{(k-2)\beta + (2k-3)\gamma}
\sum_{\substack{U \in R(p^{\beta}) \\ r_u \equiv 0 \imod{p^{\beta+\gamma}} \\ b_u \equiv s_u \equiv 0 \imod{p^{\gamma}}}} 
a \left( p^{\alpha} \sm{r_u p^{-\beta-\gamma}}{\frac{b_u}{2} p^{-\gamma}}{\frac{b_u}{2} p^{-\gamma}}{s_u p^{\beta-\gamma}} \right)
\end{equation*}
\end{theorem}

\noindent Theorem \ref{thm_Andrianov} corresponds to the level one case of Proposition 5.16 in \cite{AN3}\noindent. To prove our main results we will need to evaluate and simplify the result in Theorem \ref{thm_Andrianov} under certain circumstances. Corollaries (\ref{cor_apNRed}) - (\ref{cor_ap2N}) are the results of these efforts.  Let $\delta_x(y_1, y_2, \dots, y_n)$ equal 1 if $x$ divides each of $y_1, y_2, \dots y_n$, and zero otherwise.

\begin{cor}\label{cor_apNRed}
For $p$ prime and $N=\sm{r}{b/2}{b/2}{s}$ with $s \not\equiv 0 \pmod{p}$,
\begin{equation*}
a(p^{\delta}; N) = a(p^{\delta}N)
+\sum_{\beta=1}^{\delta} p^{(k-2)\beta}
\sum_{\substack{u=0 \\ r+bu+su^2 \equiv 0 \, (p^{\beta})}}^{p^{\beta}-1}
a \left( p^{\delta - \beta} \sm{(r+bu+su^2) p^{-\beta}}{b/2 +su}{b/2+su}{s p^{\beta}} \right)
\end{equation*}
\end{cor}

\begin{cor}\label{cor_apmI}
For $p$ prime and $m \in \mathbb{Z}^{+}$ such that $(m,p)=1$,
\begin{equation*}
a(p^{\delta}; mI) = a(mp^{\delta} I)
+
\begin{cases}
2 \displaystyle\sum_{\beta=1}^{\delta} p^{(k-2)\beta} \, a(mp^{\delta-\beta}I) & \textup{ if } p\equiv 1 \pmod{4}\\[15pt]
p^{k-2} \, a(mp^{\delta-1}I)& \textup{ if } p=2\\[6pt]
0 & \textup{ if } p\equiv 3 \pmod{4}.
\end{cases}
\end{equation*}
\end{cor}

\begin{cor}\label{cor_apN}
For $p$ prime and $N=\sm{r}{b/2}{b/2}{s}$,
\begin{multline*}
a(p; N) = a(pN)
+p^{k-2}
\left[\sum_{\substack{u=0 \\ r+bu+su^2 \equiv 0 \, (p)}}^{p-1}
\aN{(r+bu+su^2)p^{-1}}{b/2 +su}{b/2+su}{s p} 
+ \delta_p(s) \, \aN{rp}{b/2}{b/2}{sp^{-1}}
\right] \\
+ \delta_p(r,b,s)\, p^{2k-3} \, a(p^{-1}N).
\end{multline*}
\end{cor}

\begin{cor}\label{cor_ap2N}
For $p$ prime and $N=\sm{r}{b/2}{b/2}{s}$,
\begin{multline*}
a(p^2; N) = a(p^2N)
+ p^{2k-3} \, \delta_p(r,b,s) \, a(N)
+ p^{4k-6}  \, \delta_{p^2}(r,b,s) \, a(p^{-2}N)\\
+p^{k-2}
\left[\sum_{\substack{u=0 \\ r+bu+su^2 \equiv 0 \, (p)}}^{p-1}
a\left(p \sm{(r+bu+su^2)p^{-1}}{b/2 +su}{b/2+su}{s p} \right)
+ \delta_p(s) \, a \left(p \sm{rp}{b/2}{b/2}{sp^{-1}} \right)
\right]\\
+p^{2k-4}
\left[
\sum_{\substack{u=0 \\ r+bu+su^2 \equiv 0 \, (p^2)}}^{p^2-1}
\aN{(r+bu+su^2)p^{-2}}{b/2 +su}{b/2 +su}{s p^2} 
+\sum_{\substack{u=0 \\ bup+s \equiv 0 \, (p^2)}}^{p-1}
\aN{r p^2}{rup+b/2}{rup+b/2}{ru^2+(bup+s)p^{-2}} 
\right] \\
+p^{3k-5}
\left[\sum_{\substack{u=0 \\ r+bu+su^2 \equiv 0 \, (p^2) \\ b \equiv s \equiv 0 (p)}}^{p-1}
\aN{(r+bu+su^2)p^{-2}}{(b/2 +su) p^{-1}}{(b/2 +su)p^{-1}}{s}
+ \delta_p(r,b) \, \, \delta_{p^2}(s) \, \aN{r}{\frac{b}{2} p^{-1}}{\frac{b}{2} p^{-1}}{s p^{-2}}
\right].
\end{multline*}
\end{cor}


\section{Proofs}\label{sec_Proofs}
We will first need the following lemmas. The first of which examines the sets $R(p^{\beta})$, for $\beta=0,1,2$, which appear in Theorem \ref{thm_Andrianov}. 
\begin{lemma}\label{lem_RSets}
We can choose $R(p^0)$, $R(p^1)$ and $R(p^2)$ as follows:
\begin{align*}
R(p^0) &= \{ \sm{1}{0}{0}{1} \} ;\\[3pt]
R(p^1) &= \{ \sm{1}{u}{0}{1} \mid u=0,1,\cdots, p-1\} \cup \{\sm{0}{1}{-1}{0} \}; and\\[3pt]
R(p^2) &= \{ \sm{1}{u}{0}{1} \mid u=0,1,\cdots, p^2-1\} 
\cup \{ \sm{up}{1}{-1}{0} \mid u=0,\cdots, p-1\} .
\end{align*}
\end{lemma}

\begin{proof}[Proof of Lemma \ref{lem_RSets}]
$\beta=0$: Given any relatively prime pair $(u_1, u_2)$, then $(u_1, u_2)\sim(1,0) \imod{p^0}$. \\ 
$\beta=1$: We first note that no two pairs from $(1,0), (1,1), \cdots (1, p-1)$ and $(0,1)$ are equivalent under $\sim \imod{p}$. Consider a pair of relatively prime integers $(u_1, u_2)$. If $p \mid u_1$ then $p \nmid u_2$ and $(u_1, u_2) \sim (0, 1) \imod{p}$, where the relation is got by taking $a$ to be the inverse of $u_2$ modulo $p$, in (\ref{eq_EquivRel}). If $p \nmid u_1$ then $(u_1, u_2) \sim (1, u)  \imod{p}$ where $u \in \{0,1,\cdots, p-1\}$ and $u \equiv u_1^{-1} u_2 \imod{p}$.\\
$\beta=2:$ We note that no two pairs from $(1,0), (1,1), \cdots (1, p^2-1), (p, 1), (2p, 1), \cdots (p^2-p,1)$ and $(0,1)$ are equivalent under $\sim \imod{p^2}$. Consider a pair of relatively prime integers $(u_1, u_2)$. If $p^2 \mid u_1$ then $p \nmid u_2$ and $(u_1, u_2) \sim (0, 1) \imod{p^2}$, where the relation is got by taking $a$ to be the inverse of $u_2$ modulo $p^2$, in (\ref{eq_EquivRel}). If $p\mid u_1$ but $p^2 \nmid u_1$ then $p\nmid u_2$ and we can let $u_1=rp$ for some $r \in \{1,2, \cdots p-1 \}$. Then $(u_1, u_2) = (rp, u_2) \sim (up, 1) \imod{p^2}$ where $u \in \{1,\cdots, p-1\}$ and $u \equiv s r \imod{p}$, where $s$ is the inverse of $u_2$ modulo $p^2$. Finally if $p \nmid u_1$ then $(u_1, u_2) \sim (1, u)  \imod{p^2}$ where $u \in \{0,1,\cdots, p^2-1\}$ and $u \equiv u_1^{-1} u_2 \imod{p^2}$.
\end{proof}

We will the need the following two results to simplify many of the Fourier coefficients in later results. The first recalls a fact about primes congruent to 1 modulo 4.

\begin{prop}\label{prop_sumofsquares}
Let $p \equiv 1 \imod{4}$ be prime and let ${\beta}$ be a positive integer. There exist positive integers $x_{\beta}, y_{\beta}$ such that $p^{\beta}=x_{\beta}^2+y_{\beta}^2$ with $x_{\beta}$ odd, $y_{\beta}$ even and $p \nmid x_{\beta}, p \nmid y_{\beta}$.
\end{prop}

\begin{proof}
This is well known in the case ${\beta}=1$. When $\beta=2$, $x_2 = | x_1^2-y_1^2 |$ and $y_2= 2x_1y_1$ satisfy the conditions. We now inductively define $x_{n+1}$ and $y_{n+1}$, for $n>1$, in a similar manner. If $p^n = x_n^2+ y_n^2$ satisfies the terms of the proposition, then
\begin{align*}
p^{n+1} &= p^n p = (x_n^2+ y_n^2)(x_1^2+y_1^2)\\
&=|x_1x_n + y_1 y_n |^2 + | x_n y_1 - x_1 y_n | ^2\\
&=| x_1x_n - y_1 y_n |^2 + | x_n y_1 + x_1 y_n | ^2.
\end{align*}
Now $p$ cannot divide both $x_1x_n + y_1 y_n$ and $x_1x_n - y_1 y_n$. Otherwise $p$ would have to divide their sum, $2x_1x_n$, which is a contradiction. We let $x_{n+1}$ be the one of $|x_1x_n + y_1 y_n|$ and $|x_1x_n - y_1 y_n|$ which is not divisible by $p$. We then choose the corresponding $y_{n+1}$ from $|x_n y_1 \pm x_1 y_n|$.  
\end{proof}

\begin{lemma}\label{lem_STS}
Let $p \equiv 1 \imod{4}$ be prime and let ${\beta}$ be a positive integer. Then for an integer $u$ satisfying $u^2 \equiv -1 \imod{p^{\beta}}$ there exists $S \in SL_2(\mathbb{Z})$ such that
\begin{equation*}
S \, {^tS} = \mat{(1+u^2)/p^{\beta}}{u}{u}{p^{\beta}}.
\end{equation*}
\end{lemma}

\begin{proof}
We choose $x_{\beta}$, $y_{\beta}$ in accordance with Proposition \ref{prop_sumofsquares} such that $p^{\beta}=x_{\beta}^2+y_{\beta}^2$ .
Consider
\begin{equation*}
S_1 = \mat{(uy_{\beta}+x_{\beta})/p^{\beta}}{(ux_{\beta}-y_{\beta})/p^{\beta}}{y_{\beta}}{x_{\beta}}
\text{ and }
S_2 = \mat{(uy_{\beta}-x_{\beta})/p^{\beta}}{-(ux_{\beta}+y_{\beta})/p^{\beta}}{y_{\beta}}{-x_{\beta}}.
\end{equation*}
It is easy to check that $\textup{Det}(S_1)=\textup{Det}(S_2)=1$ and that
$$S_1 \, {^tS_1} = S_2 \, {^tS_2} = \mat{(1+u^2)/p^{\beta}}{u}{u}{p^{\beta}}.$$ 
We will now show that one of $S_1, S_2$ is integral. Whichever one that is, then satisfies the requirements in the lemma for S.
Consider 
$$(uy_{\beta}+x_{\beta})(uy_{\beta}-x_{\beta}) = u^2y_{\beta}^2 - x_{\beta}^2 
\equiv - y_{\beta}^2 - x_{\beta}^2 \equiv 0 \pmod{p^{\beta}}.$$
So $p^{\beta} \mid (uy_{\beta}+x_{\beta})(uy_{\beta}-x_{\beta})$. If $p \mid uy_{\beta}+x_{\beta}$ and $p \mid uy_{\beta}-x_{\beta}$ then $p \mid 2x_{\beta}$, which is a contradiction. So $p^{\beta} \mid uy_{\beta}+x_{\beta}$ or $p^{\beta} \mid uy_{\beta}-x_{\beta}$.
Similarly, by considering $(ux_{\beta}-y_{\beta})(ux_{\beta}+y_{\beta})$ we can show that $p^{\beta} \mid ux_{\beta}-y_{\beta}$ or $p^{\beta} \mid ux_{\beta}+y_{\beta}$.

Assume $p^{\beta} \mid uy_{\beta}+x_{\beta}$. If also $p^{\beta} \mid ux_{\beta}+y_{\beta}$, then $p^{\beta}$ divides their difference, i.e., $p^{\beta} \mid (x_{\beta} - y_{\beta})(u-1)$. Now $p \nmid u-1$, as $u^2 \equiv -1 \imod{p}$, so $p^{\beta} \mid x_{\beta} - y_{\beta}$. In which case
$$p^{\beta} \mid (x_{\beta} - y_{\beta})(x_{\beta} + y_{\beta})+(x_{\beta}^2 + y_{\beta}^2) = 2 x_{\beta}^2,$$
which is a contradiction. So $p^{\beta} \nmid ux_{\beta}+y_{\beta}$. Therefore $p^{\beta} \mid ux_{\beta}-y_{\beta}$. So we have proved that if $p^{\beta} \mid uy_{\beta}+x_{\beta}$ then  $p^{\beta} \mid ux_{\beta}-y_{\beta}$ also, and hence $S_1$ is integral.
Similarly, if we assume $p^{\beta} \mid uy_{\beta}-x_{\beta}$ then we can show that  $p^{\beta} \nmid ux_{\beta}-y_{\beta}$. This implies that $p^{\beta} \mid ux_{\beta}+y_{\beta}$ and that $S_2$ is integral in this case. Therefore one of $S_1, S_2$ must be integral.
\end{proof}

We now prove Corollaries \ref{cor_apNRed} - \ref{cor_apN} which we will then use to prove Theorem \ref{thm_mult}.
\begin{proof}[Proof of Corollary \ref{cor_apNRed}]
Consider Theorem \ref{thm_Andrianov}. Let $U=\sm{u_1}{u_2}{u_3}{u_4}$. Then 
$$
\mat{r_u}{b_u/2}{b_u/2}{s_u}=
\mat{r u_1^2+b u_1 u_2+s u_2^2}{r u_1 u_3 + \tfrac{b}{2} (u_1 u_4 + u_2 u_3) + su_2 u_4}{r u_1 u_3 + \tfrac{b}{2} (u_1 u_4 + u_2 u_3) + su_2 u_4}{r u_3^2+b u_3 u_4+s u_4^2}.$$
We first consider the case when $\beta =0$. By Lemma \ref{lem_RSets}, $R(p^0) = \{I\}$ and so $U= I$ is the only term to consider in the second sum. In which case $s_u = s$. The condition on the second sum that $s_u \equiv 0 \pmod {p^\gamma}$ then implies $\gamma=0$, as $s \not\equiv 0 \pmod p$. Therefore, the contribution to $a(p^{\delta}; N)$ in the $\beta=0$ case is $a(p^{\delta}N)$.

Now we consider when $\beta \geq 1$. The condition $r_u \equiv 0 \pmod{p^{\beta}}$ implies $p \mid r u_1^2+b u_1 u_2+s u_2^2$. If $p \mid u_1$ then $p \nmid u_2$ as $(u_1,u_2)=1$ and so $p \mid s$. So if $s \not\equiv 0 \pmod{p}$ then $p \nmid u_1$.
In this case $(u_1, u_2) \sim (1, u) \pmod{p^{\beta}}$, where $u \in \{0, 1, \cdots, p^{\beta}-1 \}$ with $u \equiv u_1^{-1} u_2 \pmod{p^{\beta}}$, and $u_1^{-1}$ is the inverse of $u_1$ in $\left( \mathbb{Z} / p^{\beta} \mathbb{Z} \right)^{\times}$. So if  $s \not\equiv 0 \pmod{p}$ we need only consider $U$ in the subset 
$\{ \sm{1}{u}{0}{1} \mid u= 0, 1, \cdots, p^{\beta}-1 \}$
of $R(p^{\beta})$. Note that if $U= \sm{1}{u}{0}{1}$ then 
$$\mat{r_u}{\frac{b_u}{2}}{\frac{b_u}{2}}{s_u}=
\mat{r+b u+s u^2}{ \tfrac{b}{2} + su}{\tfrac{b}{2}  + su}{s}.$$
In particular $s_u = s$ and so the condition that $s_u \equiv 0 \pmod {p^\gamma}$ implies $\gamma=0$. Thus 
Theorem \ref{thm_Andrianov} reduces to
$$a(p^{\delta}; N) = a(p^{\delta}N)
+\sum_{\beta=1}^{\delta} p^{(k-2)\beta}
\sum_{\substack{u=0 \\ r+bu+su^2 \equiv 0 \, (p^{\beta})}}^{p^{\beta}-1}
a \left( p^{\delta - \beta} \sm{(r+bu+su^2) p^{-\beta}}{b/2 +su}{b/2+su}{s p^{\beta}} \right),$$
as required.

\end{proof}

\begin{proof}[Proof of Corollary \ref{cor_apmI}]
Taking $N=mI$ in Corollary \ref{cor_apNRed} we get
\begin{equation*}
a(p^{\delta}; mI) = a(mp^{\delta}I)
+\sum_{\beta=1}^{\delta} p^{(k-2)\beta}
\sum_{\substack{u=0 \\ m+mu^2 \equiv 0 \, (p^{\beta})}}^{p^{\beta}-1}
a \left(m p^{\delta - \beta} \sm{(1+u^2) p^{-\beta}}{u}{u}{ p^{\beta}} \right)
\end{equation*}
Now $m+mu^2 \equiv 0 \imod{p^{\beta}} \Leftrightarrow u^2 \equiv -1  \imod{p^{\beta}}$ as $(m,p)=1$. If $p \equiv 3 \imod{4}$ there is no such $u$ and so $a(p^{\delta}; mI) = a(mp^{\delta}I)$. If $p=2$ the only solution to $u^2 \equiv -1  \imod{p^{\beta}}$ is when $\beta=1$, in which case $u=1$ and
$$a(p^{\delta}; mI) = a(mp^{\delta}I) + p^{k-2} \, a\bigl(m p^{\delta-1} \sm{1}{1}{1}{2}\bigr) .$$
Taking $U = \sm{1}{0}{1}{1}$ and $N=m p^{\delta-1} I$ in (\ref{for_UNTU}) we see that 
$$a\bigl(m p^{\delta-1} \sm{1}{1}{1}{2}\bigr) = a(m p^{\delta-1} I).$$
Now we examine the case when $p \equiv 1 \imod{4}$. From Lemma \ref{lem_STS} we know there exists exists $S \in SL_2(\mathbb{Z})$ such that
\begin{equation*}
S \, {^tS} = \sm{(1+u^2)/p^{\beta}}{u}{u}{p^{\beta}}.
\end{equation*}
Therefore \begin{equation*}
S \, m p^{\delta - \beta}I \, {^tS} = m p^{\delta - \beta} \sm{(1+u^2) p^{-\beta}}{u}{u}{ p^{\beta}},
\end{equation*}
and so by (\ref{for_UNTU}) we see that
$$a \left(m p^{\delta - \beta} \sm{(1+u^2) p^{-\beta}}{u}{u}{ p^{\beta}} \right) = a(m p^{\delta - \beta} I).$$
Noting that $u^2 \equiv -1  \imod{p^{\beta}}$ has two solutions when $p \equiv 1 \imod{4}$, completes the proof.
\end{proof}

\begin{proof}[Proof of Corollary \ref{cor_apN}]
We put $\delta=1$ in Theorem \ref{thm_Andrianov} and consider the cases $\alpha=1$, $\beta=1$ and $\gamma=1$ separately. \\
Case 1: $\alpha=1$. When $\alpha=1$ then $\beta= \gamma=0$ and $R(p^{\beta}) = \{ I \}$ by Lemma \ref{lem_RSets}. Therefore the contribution to $a(p;N)$ in this case is
\begin{equation*}
a(p;N)_{1} = a(pN).
\end{equation*}
Case 2: $\beta=1$. When $\beta=1$ then $\alpha= \gamma=0$ and $R(p^{\beta}) = \{ \sm{1}{u}{0}{1} \mid u=0,1,\cdots, p-1\} \cup \{\sm{0}{1}{-1}{0} \}$,
 by Lemma \ref{lem_RSets}. When $U=\sm{0}{1}{-1}{0}$ then $\sm{r_u}{b_u/2}{b_u/2}{s_u} = \sm{s}{-b/2}{-b/2}{r}$ and when $U=\sm{1}{u}{0}{1}$ then $\sm{r_u}{b_u/2}{b_u/2}{s_u} = \sm{r+bu+su^2}{b/2+su}{b/2+su}{s}$. Therefore the contribution to $a(p;N)$ in this case is
\begin{equation*}
a(p;N)_{2} = 
p^{k-2} \left[ \sum_{\substack{u=0 \\ r+bu+su^2 \equiv 0 \, (p)}}^{p-1}
\aN{(r+bu+su^2)p^{-1}}{b/2 +su}{b/2+su}{s p} 
+\aN{sp^{-1}}{-b/2}{-b/2}{rp} 
\right],
\end{equation*}
where the last term requires $s \equiv 0 \imod{p}$.
Note that by (\ref{for_ASwitch1}) we get
$$\aN{sp^{-1}}{-b/2}{-b/2}{rp} = \aN{rp}{b/2}{b/2}{sp^{-1}}.$$
Case 3: $\gamma=1$. When $\gamma=1$ then $\alpha= \beta=0$ and $R(p^{\beta}) = \{I \}$. Therefore the contribution to $a(p;N)$ in this case is
\begin{equation*}
a(p;N)_{3} = p^{2k-3} \, a(p^{-1}N),
\end{equation*}
and we require $r \equiv b\equiv s \equiv 0 \imod{p}$.\\
Combining the three cases we get the desired result.
\end{proof}

\begin{proof}[Proof of Theorem \ref{thm_mult}]
(1) Let $n \in \mathbb{Z}^{+}$. Taking $m=n$ in Corollary \ref{cor_apmI} and considering (\ref{eq_alambda}) we see that, when $(n,p)=1$,
\begin{equation}\label{eq_anlambda_pdelta}
a(nI) \lambda(p^{\delta}) = a(np^{\delta} I)
+
\begin{cases}
2 \displaystyle\sum_{\beta=1}^{\delta} p^{(k-2)\beta} \, a(np^{\delta-\beta}I) & \textup{ if } p\equiv 1 \pmod{4}\\[15pt]
p^{k-2} \, a(np^{\delta-1}I)& \textup{ if } p=2\\[6pt]
0 & \textup{ if } p\equiv 3 \pmod{4},
\end{cases}
\end{equation}
and, specifically, in the case $\delta=1$ we get that
\begin{equation}\label{eq_anlambda_p}
a(nI) \lambda(p) 
 = a(np I)
+
\begin{cases}
2 \, p^{k-2} \, a(nI) & \textup{ if } p\equiv 1 \pmod{4}\\[15pt]
p^{k-2} \, a(nI)& \textup{ if } p=2\\[6pt]
0 & \textup{ if } p\equiv 3 \pmod{4}.
\end{cases}
\end{equation}
Therefore, if $a(nI)=0$ then $a(npI)=0$, whenever $(n,p)=1$. Inductively, we can then show, using (\ref{eq_anlambda_pdelta}), that
\begin{equation}\label{eq_an_anpdelta}
a(nI)=0 \Rightarrow a(np^{\delta}I)=0,
\end{equation}
for any $\delta \in \mathbb{Z}^{+}$, whenever $(n,p)=1$. If $m=p_1^{\delta_1} p_2^{\delta_2} \cdots p_k^{\delta_k}$ for distinct primes $p_1, p_2, \cdots p_k$, then repeated use of (\ref{eq_an_anpdelta}) yields
$$a(I)=0 \Rightarrow a(p_1^{\delta_1}I)=0 \Rightarrow a(p_1^{\delta_1}p_2^{\delta_2}I)=0 \Rightarrow \dots \Rightarrow a(mI)=0,$$
as required.

(2) If $a(I)=0$ then part (1) tells us that all terms are zero, and hence the statement is trivially true. Assume $a(I)\neq0$. It then suffices to prove 
\begin{equation}\label{eq_aIampdeltaI_amIapdeltaI}
a(I) a(mp^{\delta}I) = a(mI) a(p^{\delta}I),
\end{equation}
for $(m,p)=1$, ${\delta}\in \mathbb{Z}^{+}$. 

We start with the $\delta=1$ case. From (\ref{eq_anlambda_p}) we know that
\begin{equation}\label{eq_amlambda_p}
a(mI) \lambda(p) 
 = a(mp I)
+
\begin{cases}
2 \, p^{k-2} \, a(mI) & \textup{ if } p\equiv 1 \pmod{4}\\[15pt]
p^{k-2} \, a(mI)& \textup{ if } p=2\\[6pt]
0 & \textup{ if } p\equiv 3 \pmod{4}.
\end{cases}
\end{equation}
and, when $m=1$,
\begin{equation}\label{eq_aIlambda_p}
a(I) \lambda(p) 
 = a(p I)
+
\begin{cases}
2  p^{k-2} \, a(I) & \textup{ if } p\equiv 1 \pmod{4}\\[15pt]
p^{k-2} \, a(I)& \textup{ if } p=2\\[6pt]
0 & \textup{ if } p\equiv 3 \pmod{4}.
\end{cases}
\end{equation}
We now multiply both sides of (\ref{eq_amlambda_p}) by $a(I)$ and both sides of (\ref{eq_aIlambda_p}) by $a(mI)$, and then compare the resulting expressions. In all cases we get $a(I) a(mpI)=a(mI)a(pI)$, as required.

We prove the general case by induction. Assume (\ref{eq_aIampdeltaI_amIapdeltaI}) holds for all $1 \leq t < \delta$.
From (\ref{eq_anlambda_pdelta}) we know that
\begin{equation}\label{eq_amlambda_pdelta}
a(mI) \lambda(p^{\delta}) = a(mp^{\delta} I)
+
\begin{cases}
2 \displaystyle\sum_{\beta=1}^{\delta} p^{(k-2)\beta} \, a(mp^{\delta-\beta}I) & \textup{ if } p\equiv 1 \pmod{4}\\[15pt]
p^{k-2} \, a(mp^{\delta-1}I)& \textup{ if } p=2\\[6pt]
0 & \textup{ if } p\equiv 3 \pmod{4},
\end{cases}
\end{equation}
and, when $m=1$,
\begin{equation}\label{eq_aIlambda_pdelta}
a(I) \lambda(p^{\delta}) = a(p^{\delta} I)
+
\begin{cases}
2 \displaystyle\sum_{\beta=1}^{\delta} p^{(k-2)\beta} \, a(p^{\delta-\beta}I) & \textup{ if } p\equiv 1 \pmod{4}\\[15pt]
p^{k-2} \, a(p^{\delta-1}I)& \textup{ if } p=2\\[6pt]
0 & \textup{ if } p\equiv 3 \pmod{4},
\end{cases}
\end{equation}
We now multiply both sides of (\ref{eq_amlambda_pdelta}) by $a(I)$ and both sides of (\ref{eq_aIlambda_pdelta}) by $a(mI)$, and then compare the resulting expressions. Note that $a(I) a(mp^{\delta-\beta}I) = a(mI) a(p^{\delta-\beta}I)$ for all $1 \leq \beta \leq \delta$ by the induction hypothesis.  Therefore, in all cases we get $a(I) a(mp^{\delta}I)=a(mI)a(p^{\delta}I)$, as required.

(3) Taking $N=p^rI$, for $r \geq 1$, in Corollary \ref{cor_apN} and combining with (\ref{eq_alambda}) we get that
\begin{multline}\label{eq_aprIlambda_p}
a(p^rI) \lambda(p) = \\
a(p^{r+1}I) + p^{k-2}
\left[\sum_{u=0}^{p-1}
\aN{p^{r-1} (1+u^2)}{p^r u}{p^r u}{p^{r+1}} 
+ \aN{p^{r+1}}{0}{0}{p^{r-1}}
\right]
+ p^{2k-3} \, a(p^{r-1}I)
\end{multline}
Multiplying both sides of (\ref{eq_aprIlambda_p}) by $a(I)$ and both sides of (\ref{eq_aIlambda_p}), which comes from Corollary \ref{cor_apmI} with $\delta=m=1$ so still holds, by $a(p^{r}I)$, and comparing the resulting expressions yields
\begin{multline}\label{for_apr1_1}
a(I) a(p^{r+1} I) = a(pI) a(p^r I) - p^{2k-3} a(I) a(p^{r-1}I)
- p^{k-2} a(I)\\
\times
 \left[2 \, \aN{p^{r+1}}{0}{0}{p^{r-1}}
 + \sum_{u=1}^{p-1}
\aN{p^{r-1} (1+u^2)}{p^r u}{p^r u}{p^{r+1}} 
- a(p^r I) \times
\begin{cases}
2  & \textup{ if } p\equiv 1 \pmod{4},\\
1 & \textup{ if } p=2.
\end{cases}
\right]
\end{multline}
If $p=2$ then $u=1$ is the only contribution to the sum, and $\aN{p^{r-1} (1+u^2)}{p^r u}{p^r u}{p^{r+1}} = a\left( p^r \sm{1}{1}{1}{2} \right)$. Now
$\sm{1}{0}{1}{1} I \sm{1}{1}{0}{1} =  \sm{1}{1}{1}{2}$ and so, by (\ref{for_UNTU}),
$a\left( p^r \sm{1}{1}{1}{2} \right) = a(p^r I)$. Therefore the result is proven in the case $p=2$.

We now consider the case when $p \equiv 1 \pmod{4}$. When $u^2 \equiv -1 \pmod {p}$ we know from Lemma \ref{lem_STS}, with $\beta =1$, that there exists $S \in SL_2(\mathbb{Z})$ such that
\begin{equation*}
S \, p^rI \, {^tS} = p^r \mat{(1+u^2)/p}{u}{u}{p},
\end{equation*}
and so by (\ref{for_UNTU}),
$\aN{p^{r-1} (1+u^2)}{p^r u}{p^r u}{p^{r+1}} = a(p^r I)$. There are two solutions to $u^2 \equiv -1 \pmod {p}$  in the range $1 \leq u \leq p-1$, when $p \equiv 1 \pmod{4}$, and none when $p \equiv 3 \pmod{4}$. Therefore, when $p \neq 2$, (\ref{for_apr1_1}) becomes,
\begin{multline*}\label{for_apr1_2}
a(I) a(p^{r+1} I) = a(pI) a(p^r I) - p^{2k-3} a(I) a(p^{r-1}I)
- p^{k-2} a(I)\\
\times
 \left[2 \, \aN{p^{r+1}}{0}{0}{p^{r-1}}
 + \sum_{\substack{u = 1 \\ u^2 \not\equiv -1 \imod{p}}}^{p-1}
\aN{p^{r-1} (1+u^2)}{p^r u}{p^r u}{p^{r+1}} 
\right].
\end{multline*}
Using (\ref{for_UNTU}) again and noting that 
$$ \sm{-1}{1}{0}{1} p^r \sm{(1+u^2)/p}{u}{u}{p} \sm{-1}{0}{1}{1} = p^r \sm{(1+(p-u)^2)/p}{p-u}{p-u}{p},$$
tells us that 
$$\aN{p^{r-1} (1+(p-u)^2)}{p^r (p-u)}{p^r (p-u)}{p^{r+1}} = (-1)^k \, \aN{p^{r-1} (1+u^2)}{p^r u}{p^r u}{p^{r+1}},$$
which completes the proof.
\end{proof}

We now prove Corollary \ref{cor_ap2N} which we use to prove Theorems \ref{thm_lambdaAI} and \ref{thm_lambdaA0}.

\begin{proof}[Proof of Corollary \ref{cor_ap2N}]
We put $\delta=2$ in Theorem \ref{thm_Andrianov} and consider separately the six cases covering each of the possible values for $(\alpha, \beta, \gamma)$. \\
Case 1: $(\alpha, \beta, \gamma)=(2,0,0)$.
In this case $R(p^{\beta}) = \{ I \}$ by Lemma \ref{lem_RSets}. Therefore the contribution to $a(p;N)$ in this case is
\begin{equation*}
a(p;N)_{1} = a(p^2 N).
\end{equation*}
Case 2: $(\alpha, \beta, \gamma)=(0,2,0)$.
In this case $R(p^{\beta}) = R(p^2) = \{ \sm{1}{u}{0}{1} \mid u=0,1,\cdots, p^2-1\} 
\cup \{ \sm{up}{1}{-1}{0} \mid u=0,\cdots, p-1\} $
 by Lemma \ref{lem_RSets}. 
When 
$U=\sm{1}{u}{0}{1}$ then $\sm{r_u}{b_u/2}{b_u/2}{s_u} = \sm{r+bu+su^2}{b/2+su}{b/2+su}{s}$, and when $U=\sm{up}{1}{-1}{0}$ then $\sm{r_u}{b_u/2}{b_u/2}{s_u} = \sm{r(up)^2+bup+s}{-rup-b/2}{-rup-b/2}{r}$.
Therefore the contribution to $a(p;N)$ in this case is
\begin{multline*}
a(p;N)_{2} = 
p^{2k-4} \left[
\sum_{\substack{u=0 \\ r+bu+su^2 \equiv 0 \, (p^2)}}^{p^2-1} 
\aN{(r+bu+su^2)p^{-2}}{b/2+su}{b/2+su}{sp^2}
\right. \\ \left.
+
\sum_{\substack{u=0 \\ bup+s \equiv 0 \, (p^2)}}^{p-1} 
\aN{(r(up)^2+bup+s)p^{-2}}{-rup-b/2}{-rup-b/2}{rp^2}
\right].
\end{multline*}
Note that by (\ref{for_ASwitch1}) we get
$$\sm{(r(up)^2+bup+s)p^{-2}}{-rup-b/2}{-rup-b/2}{rp^2}=\sm{rp^2}{rup+b/2}{rup+b/2}{(r(up)^2+bup+s)p^{-2}}.$$
Case 3: $(\alpha, \beta, \gamma)=(0,0,2)$. In this case $R(p^{\beta}) = \{ I \}$ by Lemma \ref{lem_RSets}. Therefore the contribution to $a(p;N)$ in this case is
\begin{equation*}
a(p;N)_{3} = p^{4k-6} \, a(p^{-2} N),
\end{equation*}
where we require $r \equiv b \equiv s \equiv 0 \imod{p^2}$.\\
Case 4: $(\alpha, \beta, \gamma)=(1,1,0)$. In this case $R(p^{\beta}) = R(p^1) = \{ \sm{1}{u}{0}{1} \mid u=0,1,\cdots, p-1\} \cup \{\sm{0}{1}{-1}{0} \}$ by Lemma \ref{lem_RSets}.
When $U=\sm{0}{1}{-1}{0}$ then $\sm{r_u}{b_u/2}{b_u/2}{s_u} = \sm{s}{-b/2}{-b/2}{r}$ and when $U=\sm{1}{u}{0}{1}$ then $\sm{r_u}{b_u/2}{b_u/2}{s_u} = \sm{r+bu+su^2}{b/2+su}{b/2+su}{s}$.
Therefore the contribution to $a(p;N)$ in this case is
\begin{equation*}
a(p;N)_{4} = p^{k-2} \left[\sum_{\substack{u=0 \\ r+bu+su^2 \equiv 0 \, (p)}}^{p-1} 
 a\left(p \sm{(r+bu+su^2)p^{-1}}{b/2 +su}{b/2 +su}{s p} \right)
+a\left(p \sm{sp^{-1}}{-b/2}{-b/2}{rp} \right) 
\right],
\end{equation*}
where we require $s\equiv 0 \imod{p}$ in the last term. We then use (\ref{for_ASwitch1}) on the last term.\\
Case 5: $(\alpha, \beta, \gamma)=(1,0,1)$. In this case $R(p^{\beta}) = \{ I \}$ by Lemma \ref{lem_RSets}. Therefore the contribution to $a(p;N)$ in this case is
\begin{equation*}
a(p;N)_{5} = p^{2k-3} \;  a(N),
\end{equation*}
where we require $r \equiv b \equiv s \equiv 0 \imod{p}$. \\
Case 6: $(\alpha, \beta, \gamma)=(0,1,1)$.  In this case $R(p^{\beta}) = R(p^1) = \{ \sm{1}{u}{0}{1} \mid u=0,1,\cdots, p-1\} \cup \{\sm{0}{1}{-1}{0} \}$ by Lemma \ref{lem_RSets}.
When $U=\sm{0}{1}{-1}{0}$ then $\sm{r_u}{b_u/2}{b_u/2}{s_u} = \sm{s}{-b/2}{-b/2}{r}$ and when $U=\sm{1}{u}{0}{1}$ then $\sm{r_u}{b_u/2}{b_u/2}{s_u} = \sm{r+bu+su^2}{b/2+su}{b/2+su}{s}$.
Therefore the contribution to $a(p;N)$ in this case is
\begin{equation*}
a(p;N)_{6} = p^{3k-5} 
 \left[\sum_{\substack{u=0 \\ r+bu+su^2 \equiv 0 \, (p^2) \\ b \equiv s \equiv 0  \, (p)}}^{p-1} 
 a\left(p^{-1} \sm{(r+bu+su^2)p^{-1}}{b/2 +su}{b/2 +su}{s p} \right)
+a\left(p^{-1} \sm{sp^{-1}}{-b/2}{-b/2}{rp} \right) 
\right],
\end{equation*}
where we require $s \equiv 0 \imod{p^2}$ and $r \equiv b \equiv  0 \imod{p}$ in the last term. We then use (\ref{for_ASwitch1}) on the last term to get the desired format.\\
\end{proof}

\begin{proof}[Proof of Theorem \ref{thm_lambdaAI}]
We take $N = I$ in Corollaries \ref{cor_apN} and \ref{cor_ap2N} and simplify. From Corollary \ref{cor_apN} we get that 
$$a(p; I) = a(pI)
+p^{k-2} \,
\sum_{\substack{u=0 \\ u^2 \equiv -1 \, (p)}}^{p-1}
\aN{(1+u^2)p^{-1}}{u}{u}{p}.$$
If $p\equiv 3 \imod{4}$ then $u^2 \equiv -1 \imod{p}$ has no solution and so in this case $a(p; I) = a(pI)$. If $p\equiv 1 \imod{4}$ there are exactly 2 distinct solutions to $u^2 \equiv -1 \imod{p}$ for $0 \leq u \leq p-1$. Also, in this case, Lemma \ref{lem_STS} tells us that there exists $S \in SL_2(\mathbb{Z})$ such that
$$S \, {^t}S = \sm{(1+u^2)p^{-1}}{u}{u}{p},$$
and so 
$$\aN{(1+u^2)p^{-1}}{u}{u}{p} = a(S \, I \, {^t}S) = a(I)$$
by (\ref{for_UNTU}). Therefore when $p\equiv 1 \imod{4}$ we have
$$a(p; I) = a(pI) +2 \, p^{k-2} \, a(I).$$
If $p=2$ then $u=1$ is the only solution to $u^2 \equiv -1 \imod{p}$ and
$$\aN{(1+u^2)p^{-1}}{u}{u}{p} = \aN{1}{1}{1}{2} = a \left( \sm{1}{0}{1}{1} \, I \, \sm{1}{1}{0}{1} \right) = a(I)$$
by (\ref{for_UNTU}). So if $p=2$ then $a(p; I) = a(pI) + p^{k-2} \, a(I)$.
Now using the fact that $a(I) \, \lambda(p) = a(p;I)$ by (\ref{eq_alambda}), and noting that $a(I)=1$, yields the result for $\lambda(p)$.

Letting $N=I$ in Corollary \ref{cor_ap2N} we get that
\begin{multline*}
a(p^2; I) \\
= a(p^2 \, I)
+p^{k-2} \sum_{\substack{u=0 \\ u^2 \equiv -1 \, (p)}}^{p-1}
a\left(p \sm{(1+u^2)p^{-1}}{u}{u}{p} \right)
+p^{2k-4}
\sum_{\substack{u=0 \\ u^2 \equiv -1 \, (p^2)}}^{p^2-1}
\aN{(1+ u^2)p^{-2}}{u}{u}{p^2} .
\end{multline*}
Let $\beta \in \{1,2\}$. If $p\equiv 3 \imod{4}$ then $u^2 \equiv -1 \imod{p^{\beta}}$ has no solution and so in this case $a(p^2; I) = a(p^2I)$.
If $p\equiv 1 \imod{4}$ there are exactly 2 distinct solutions to $u^2 \equiv -1 \imod{p^{\beta}}$ for $0 \leq u \leq p^{\beta}-1$. Also, in this case, Lemma \ref{lem_STS} tells us that there exists $S \in SL_2(\mathbb{Z})$ such that
$$S \, {^t}S = \sm{(1+u^2)p^{-\beta}}{u}{u}{p^{\beta}}.$$
Therefore
$$\aN{(1+u^2)p^{-2}}{u}{u}{p^2} = a(S \, I \, {^t}S) = a(I)$$
and
$$a\left(p \sm{(1+u^2)p^{-1}}{u}{u}{p} \right) = a(S \, pI \, {^t}S) = a(pI),$$
by (\ref{for_UNTU}). So when $p\equiv 1 \imod{4}$ we have
$$a(p^2; I) = a(p^2I) +2 \, p^{k-2} \, a(pI) +2 \, p^{2k-4} \, a(I).$$
If $p=2$ then $u^2 \equiv -1 \imod{p^2}$ has no solutions and $u=1$ is the only solution to $u^2 \equiv -1 \imod{p}$. In the latter case
$$a\left(p \sm{(1+u^2)p^{-1}}{u}{u}{p} \right) = a \left( p \sm{1}{1}{1}{2} \right) = a \left( \sm{1}{0}{1}{1} \, pI \, \sm{1}{1}{0}{1} \right) = a(pI)$$
by (\ref{for_UNTU}). So if $p=2$ then $a(p^2; I) = a(p^2I) + p^{k-2} \, a(pI)$.
Now using the fact that $a(I) \, \lambda(p^2) = a(p^2;I)$ by (\ref{eq_alambda}), and noting that $a(I)=1$, completes the proof.
\end{proof}

\begin{proof}[Proof of Theorem \ref{thm_lambdaA0}]
We take $N = 0 = \sm{0}{0}{0}{0}$ in Corollaries \ref{cor_apN} and \ref{cor_ap2N}. From Corollary \ref{cor_apN} we get that
$$a(p;0) = a(0) + p^{k-2} \left[ p \, a(0) + a(0) \right] + p^{2k-3} \,a(0).$$
Therefore, applying (\ref{eq_alambda}) with $N=0$ and noting that $a(0) \neq 0$,
$$\lambda(p) = 1+p^{k-2} \left[ p  + 1 \right] + p^{2k-3},$$
as required.
Similarly, from Corollary \ref{cor_ap2N} we get that
\begin{multline*}
a(p^2;0) = a(0) + p^{2k-3} \, a(0) + p^{4k-6} \, a(0) +  p^{k-2} \left[p \, a(0) + a(0) \right]\\
 + p^{2k-4} \left[p^2 \, a(0) + p \, a(0) \right] + p^{3k-5} \left[ p \, a(0) + a(0) \right],
\end{multline*}
and so
$$\lambda(p^2) 
= 1 + p^{2k-3} + p^{4k-6}  +  p^{k-2} \left[p + 1 \right]
 + p^{2k-4} \left[p^2  + p \right] + p^{3k-5} \left[ p  + 1 \right],$$
as required.
\end{proof}


\section{Concluding Remarks}\label{sec_cr}
The results in this paper rely heavily on Andrianov's formula, Theorem \ref{thm_Andrianov}, for relating the Fourier coefficients of $T(p^{\delta}) F$ to the Fourier coefficients of $F$. Evdokimov \cite{Ev} generalizes this formula to degree 2 Siegel eigenforms with level. This generalization can also be derived from Andrianov's results in \cite[Proposition 5.16]{AN3}.   Using Evdokimov's result, we plan to investigate if the results of this paper can be generalized to forms with level. This will be the subject of a forthcoming paper. Preliminary investigations suggest that, at the very least, multiplicative relations between Fourier coefficients, analogous to those in Theorem \ref{thm_mult}, can be found for degree 2 Siegel eigenforms of level 2, and for degree 2 Siegel eigenforms constructed from half-integral Igusa theta constants, which are of level 8.  


\section*{Acknowledgements}
This work was supported by a grant from the Simons Foundation (\#353329, Dermot McCarthy).


\end{document}